\documentclass[12pt]{article}
\usepackage{amsmath,amsfonts,amssymb,amsthm}
\usepackage{bbm}
\usepackage{mathtools}
\usepackage{graphicx}
\mathtoolsset{showonlyrefs}

\usepackage{tabularx}
\usepackage{color}
\definecolor{darkgreen}{rgb}{0,0.55,0}
\newcommand{\grad}{\nabla}

\renewcommand{\div}{\grad\cdot}

\newcommand{\R}{\mathbbm{R}}

\DeclareMathOperator{\dist}{dist}

\DeclareMathOperator{\spt}{spt}

\usepackage{sectsty}

\sectionfont{\large}
\subsectionfont{\normalsize}
\subsubsectionfont{\normalsize}
\paragraphfont{\normalsize}

\newcommand{\eps}{\varepsilon}

\def\XXint#1#2#3{{\setbox0=\hbox{$#1{#2#3}{\int}$ }
\vcenter{\hbox{$#2#3$ }}\kern-.59\wd0}}

\newtheorem{theorem}{Theorem}
\newtheorem{lemma}{Lemma}
\newtheorem{remark}{Remark}

\newtheorem{cor}{Corollary}

\begin{document}
\phantom{ }
\vspace{4em}

\begin{flushleft}
{\large \bf Vortex dynamics for 2D Euler flows with unbounded vorticity}\\[2em]
{\normalsize \bf Stefano Ceci and Christian Seis}\\[0.5em]
\small Institut f\"ur Analysis und Numerik,  Westf\"alische Wilhelms-Universit\"at M\"unster, Germany.\\
E-mails: ceci@wwu.de, seis@wwu.de \\[3em]
{\bf Abstract:} 
It is well-known that the dynamics of vortices in an ideal incompressible two-dimensional fluid contained in a bounded not necessarily simply connected smooth domain is  described by the Kirchhoff--Routh point vortex system. In this paper, we revisit the classical problem of how well solutions to the Euler equations approximate these vortex dynamics and extend previous rigorous results to the case where the vorticity field is unbounded. More precisely, we establish estimates for the $2$-Wasserstein distance between the vorticity  and the empirical measure associated with the point vortex dynamics. In particular, we derive an estimate on the order of weak convergence of the Euler solutions to the solutions of the point vortex system.
\end{flushleft}

\vspace{2em}

\section{Introduction}

The aim of this paper is to study the motion and interaction of vortices in an ideal incompressible two-dimensional fluid filling up a bounded domain with holes. More precisely, we study the evolution by the Euler equations of vortex patches with possibly unbounded vorticity  under mild concentration assumptions on the initial configuration. In our main result, we quantify the convergence of solutions to the Euler equations towards  a system of interacting point vortices.

The investigation of the dynamics of such idealized  vortex systems goes back to the pioneering work of Helmholtz in the middle of the 19th century \cite{Helmholtz1858}. Helmholtz (implicitly) introduced the point  vortex system in the full space and derived some of its most fundamental properties. In his lectures on mathematical physics \cite{Kirchhoff76}, Kirchhoff later demonstrated that the system is of Hamiltonian form, which was  extended to the case of bounded domains first by Routh \cite{Routh81} and later by Lin \cite{Lin43}. The system of equations governing the motion of point vortices in a bounded region is today accordingly  referred to as the Kirchhoff--Routh system.

The first rigorous connection between the Euler equations and the  point vortex dynamics was established by Marchioro and Pulvirenti \cite{MarchioroPulvirenti83}. The authors consider vortex patches that are initially confined in small regions in the plane and they show that the size of these regions can be suitably controlled during the evolution. The argument exploits the symmetry properties of the two-dimensional Biot--Savart kernel and relies on the regularity of the velocity field away from the vortex patches.  Once the preservation of vorticity concentration   is proved, the convergence towards the point vortex system follows immediately. The Marchioro--Pulvirenti method was subsequently gradually refined in \cite{MarchioroPulvirenti93, CapriniMarchioro15,ButtaMarchioro18}. 

In situations in which the Biot--Savart kernel lacks certain symmetry features, the method from \cite{MarchioroPulvirenti83} seems to fail. This is the case, for instance, for the axisymmetric Euler equations without swirl, which describe the motion and interaction of vortex rings, or for the lake equations, which serve as a shallow water model in the regime of small Froude numbers. In such systems, the dynamics of   point vortices could be derived via energy expansion methods  \cite{BenedettoCagliotiMarchioro00, ButtaMarchioro19,DekeyserVanSchaftingen19}; see also \cite{Turkington87} for an analogous result for the two-dimensional Euler equations. These techniques, however, seem not to be suitable for capturing the interaction of vortices. Indeed, the results in \cite{Turkington87, BenedettoCagliotiMarchioro00, DekeyserVanSchaftingen19} are restricted to single vortices while \cite{ButtaMarchioro19} describes the evolution of travelling wave type solutions. In particular, what remains an open problem until today is the rigorous derivation of the leapfrogging dynamics of interacting vortex rings, which was already announced  by Helmholtz \cite{Helmholtz1858} and later explicitly formulated by Hicks \cite{Hicks22} under the metaphorical title \emph{The mutual threading of vortex rings}.

Quite recently,  Davila, Del Pino, Musso and Wei \cite{DavilaDelPinoMussoWei18} proposed a third ansatz for constructing vortex solutions to the two-dimensional Euler equations  that contain  precise  information on the vortex cores. Their approach relies on what is commonly referred to as the  gluing method --- a method that was previously successfully applied to a number of desingularization problems, including  concentration phenomena along curves for the Schr\"odinger equation \cite{delPinoKowalczykWei07} and non-affine solutions to the stationary Allen--Cahn equation in large dimensions \cite{delPinoKowalzykWei11}.

In the present paper, we reconsider the method of Marchioro and Pulvirenti and extend the results on vortex dynamics for the two-dimensional Euler equations to the case of unbounded vorticity fields. Such a result was, in a certain sense, 
already foreshadowed by a recent contribution of Caprini and Marchioro \cite{CapriniMarchioro15}, in which the authors prove the connection between the Euler equations and the point vortex model in situations where the vorticity is bounded but of arbitrary large amplitude. To be more specific, we consider vorticity fields with a (suitably small)  control in $L^p$ for some $p>2$. Our main result, Theorem \ref{th2}, states that any solution to the Euler equation whose vorticity initially concentrates around a finite collection of points, remains (in a scale-independent time interval) concentrated  at the same scale  and the center of concentration can be chosen as the vortex points that evolve by the Kirchhoff--Routh model.  This extension is particularly remarkable, as our integrability setting falls outside the range of known well-posedness results in $L^{\infty}$ or $\text{BMO}$ established by Yudovich \cite{Yudovich63}  and Vishik \cite{Vishik99}; see also \cite{Loeper06,AzzamBedrossian15} for alternative and simplified proofs. Yet, it is crucial to note that the range of integrability exponents we are dealing with is precisely the one in which the fluid velocity is (H\"older) continuous and thus bounded as a consequence of Calder\'on--Zygmund and Sobolev inequalities.  

As a measure of concentration, we consider the Wasserstein distance $W_2$ between the vorticity field and the empirical measure associated with the point vortices. Because Wasserstein distances metrize weak convergence, cf.~\cite[Theorem 7.12]{Villani03},   our result provides an estimate on the order of weak convergence of (possibly non-unique) solutions to the Euler equations towards the unique solution to the Kirchhoff--Routh model. We remark that the Wasserstein distance played a crucial role already in the original work of Marchioro and Pulvirenti \cite{MarchioroPulvirenti83}, even though it has been considered rather as an auxiliary  second moment function in there, cf.~\eqref{C1}.

Weak concentration measures were recently introduced in the context of the three-dimensional Euler vortex filament dynamics   in \cite{JerrardSeis17}. In this work, the flat norm is used to measure the distance between the three dimensional vorticity vector from a curve in $\R^3$.  Note that in two dimensions, it is nothing but the Wasserstein (or Kantorovich--Rubinstein) distance $W_1$, or, equivalently, the  norm associated with the negative Sobolev space $W^{-1,1}$.
Weak concentration measures  proved to be suitable tools already in the context of the dynamics of Ginzburg--Landau-type vortices, see for example \cite{CollianderJerrard98,JerrardSpirn15}. 

In a certain sense, our concentration estimate can be considered as a stability result for the two-dimensional Euler equations. Indeed, our main estimate provides a control of the distance between the vorticity field and the empirical measure associated with the point vortex dynamics --- which can be formally considered as a singular solution to the Euler equations   --- in terms of the distance of the corresponding initial configurations. 
Hence, in some sense, we provide a stability-type estimate between a weak  and an even weaker solution, that, however, carries more structure. This type of result thus differs from what is usually considered in the literature, like   estimates for weak-strong couples of solutions, e.g., \cite{Wiedemann17,SerfatyDuerinckx18}, or the propagation of compactness, e.g., \cite{CrippaSemenovaSpirito18}.

We remark that until today, general stability estimates for the Euler equations are not available. To the best of our knowledge, the estimates that are  closest to stability estimates are those by Loeper \cite{Loeper06}, who reproves Yudovich's uniqueness result by using optimal transportation techniques.  (He mainly works with the Wasserstein distance $W_2$ as well. See also \cite{Seis20} for a reformulation.)  Even for  linear transport equations with general Sobolev vector fields, stability estimates were obtained only recently in  \cite{Seis17,Seis18}; see also \cite{CrippaDeLellis08} for the corresponding results for Lagrangian flows.

In the following section, we will present the precise mathematical setting and state our main result, the proofs of which can be found in Section \ref{S3}.

\section{Mathematical setting and result}\label{S2}

We consider the Euler  equations in a bounded smooth  connected domain $\Omega$ in $\R^2$, where we allow for the presence of a finite number of disjoint holes. We denote by $\Gamma_1,\ldots,\Gamma_M$ the inner boundaries of $\Omega$, and by $\Gamma_0$ the exterior contour. All boundaries are assumed to be smooth closed curves (of positive length), and disjoint one to another.

The evolution can be stated in terms of the scalar vorticity field $\omega=\omega(t,x)\in\R$, which is simply transported by the flow of the fluid velocity $u=u(t,x)\in\R^2$ and thus mathematically described by the transport equation
\begin{equation}\label{eq1}
	\partial_t\omega+u\cdot\nabla\omega=0\quad \text{in }\Omega.
\end{equation}
We suppose that the fluid is incompressible, which translates into the mathematical condition $\div u=0$ in $\Omega$. By assuming no-penetration boundary conditions, that is, $u\cdot \nu=0$ on $\partial\Omega$, where $\nu$ is the outer normal, we ensure that there is no flow across the boundary of the domain.

It is a well-known consequence of Kelvin's circulation theorem that the mean tangential velocity along any boundary component is preserved by any sufficiently regular solution to the momentum equation. That is, there exists circulation constants $\gamma_1,\dots, \gamma_M\in\mathbb{R}$, which are determined by the initial velocity, such that \begin{equation*}
	\int_{\Gamma_m} u(t,x)\cdot\tau(x)\,d\sigma=\gamma_m,
\end{equation*}
for every time $t>0$ and any $m=1,\ldots,M$.  Here, $\tau$ denotes the unit vector tangential to the boundary, by convention set in counterclockwise direction. In fact, the validity of Kelvin's theorem in a framework that is less regular than the one considered in the present paper has been recently investigated in \cite{IftimieLopesNussenzveig20}.

%
%

While the vorticity can be computed from the velocity field by taking the curl, that is, $\omega=\partial_1 u_2-\partial_2 u_1$, the velocity field can be reconstructed from the vorticity and the circulation numbers with the help of the Biot--Savart law. The formulation requires some preparations. We consider the harmonic measures  $w_1,\ldots,w_M:\Omega\to\R$ solving  
\begin{equation*}
	\begin{cases}
	\Delta w_m=0 &\text{in }\Omega, \\
	w_m=1 &\text{on }\Gamma_m, \\
	w_m=0 &\text{on }\Gamma_l, \; l\neq m, \\
	w_m=0 &\text{on }\Gamma_0,
	\end{cases}
\end{equation*}
and let for $\xi_1,\ldots,\xi_M:\Omega\to\mathbb{R}^2$ be the unique divergence-free and curl-free vector fields satisfying $\xi_m\cdot\nu=0$ on $\partial\Omega$ and
\begin{equation*}
	\int_{\Gamma_l} \xi_m\cdot\tau\,d\sigma=\begin{cases}
	1 &\text{if }l=m ,\\
	0 &\text{if }l\neq m,
	\end{cases}
\end{equation*}
for any $l,m=1,\ldots,M$.
The Biot--Savart law then reads
\begin{equation}\label{eq26}
u(t,x)=K*\omega(t,x)+\nabla^\perp\eta(t,x)+\sum_{m=1}^M\left(\int_{\Omega} w_m(z)\omega(t,z)\,dz+\gamma_m\right)\xi_m(x).
\end{equation}
Here $K$ is the rotated gradient of the Newtonian potential $G(z) = -\frac1{2\pi}\log|z|$, that is, $K = -\nabla^\perp G$ or
\begin{equation*}
K(z)=\frac{1}{2\pi}\frac{z^\perp}{|z|^2} \quad \text{for }z^\perp=(-z_2,z_1),
\end{equation*}
and $\eta$ is the harmonic extension of  $G\ast \omega$ (this is the stream function in $\R^2$) restricted to the boundary,
\begin{equation}\label{eq7}
\begin{cases}
-\Delta\eta(t,\cdot)=0 & \text{in }\Omega, \\
\eta(t,\cdot)=G*\omega(t,\cdot) & \text{on }\partial\Omega.
\end{cases}
\end{equation}

We finally equip the Euler vorticity equation with an initial condition,
\[
\omega(0) = \bar \omega\quad\mbox{in }\Omega.
\]

Throughout this article, we assume that the vorticity field belongs to the Le\-besgue space $L^{\infty}((0,T);L^p(\Omega))$ for some $p>2$. Thanks to Calder\'on--Zygmund theory, the associated velocity field is thus Sobolev regular in the spatial variable,   $u\in L^{\infty}((0,T);W^{1,p}(\Omega))$, and then bounded, $u\in L^{\infty}((0,T)\times \Omega)$, by Sobolev embedding. It follows that the product $u\omega$ is integrable in $\Omega$ and thus, the transport equation \eqref{eq1} can be interpreted in the sense of distributions. 
It is known that such a solution exists for every $p\ge 1$, see e.g. \cite{IftimieLopesNussenzveig20}, however still nothing is known about uniqueness if $p<\infty$. Moreover, every solution is renormalized in the sense of DiPerna and Lions \cite{DiPernaLions89}, that is
\begin{equation*}
\partial_t\beta(\omega)+u\cdot\nabla\beta(\omega)=0\quad\mbox{in }\Omega
\end{equation*}
for every bounded $\beta\in C^1(\R)$ that  vanishes near $0$ and has suitable decay properties at infinity. Furthermore, from the theory in \cite{DiPernaLions89} (and \cite{Ambrosio04}) it follows that the vorticity is transported by the (regular) Lagrangian flow $\phi$ of the velocity field $u$, that is
\[
	\omega(t,\phi_t(x))=\bar{\omega}(x),
\]
where $\phi = \phi_t(x)\in\R^2$ solves the ordinary differential equation
\begin{equation}\label{C13}
    \partial_t\phi_t(x)=u(t,\phi_t(x)), \quad \phi_0(x)=x,
\end{equation}
see \cite{Ambrosio04} for a precise definition in the case of rough velocity fields.

We shall now make our choice of initial data more specific in order to be able to capture the vortex dynamics that we described in the introduction. We suppose that the vorticity can be decomposed into $N$ separated patches, that is, we suppose that
\[
\bar \omega = \sum_{i=1}^N \bar \omega_i,
\]
and the patches are disjointly supported and not touching the  boundary,
\begin{equation}\label{C9}
	\min\limits_{i\neq j}\text{dist}\big(\spt \bar{\omega}_i,\spt \bar{\omega}_j\big)\ge\delta, \quad \min\limits_{i}\text{dist}\big(\spt\bar{\omega}_i,\partial\Omega\big)\ge\delta,
\end{equation}
for some $\delta>0$. We also assume that every patch has fixed sign, that is for every $i$ it must hold either $\bar{\omega}_i \ge 0$ or $\bar{\omega}_i \le 0$. At any later time, we may write
\[
\omega(t) = \sum_{i=1}^N \omega_i(t),
\]
where $\omega_i$ is the unique (cf.~\cite{DiPernaLions89}) solution to the linear transport equation $\partial_t \omega_i + u\cdot \grad \omega_i=0$ with the initial datum $\omega_i(0) = \bar \omega_i$, and thus, $\omega_i$ transported by the flow $\phi$. In particular, because $u$ is bounded, there exists a maximal time $T\in (0,\infty]$ such that the supports of the vortex patches $\omega_1,\dots, \omega_N$ remain separated and in distance to the boundary in the sense that
\begin{equation}\label{C5}
	\min\limits_{i\neq j}\text{dist}\big(\spt {\omega}_i(t),\spt {\omega}_j(t)\big)\ge\frac{\delta}2, \quad \min\limits_{i}\text{dist}\big(\spt {\omega}_i(t) ,\partial\Omega\big)\ge\frac\delta2
\end{equation}
for all $t\in[0,T)$.

We denote the  intensity of the $i$th vortex  patch by $a_i$. It is preserved by the evolution and given by
\begin{equation}\label{C4}
a_i=\int_{\Omega}\bar \omega_i\,dx  = \int_{\Omega} \omega_i(t)\, dx.
\end{equation} 
We suppose that each vortex patch is initially concentrated around a certain point $\bar Y_i$ in $\Omega$ in the sense that
\begin{equation}\label{eq23}
W_2\left(\frac{\bar{\omega}_i}{a_i},\,\delta_{\bar{Y}_i}\right)\le\varepsilon,
\end{equation}
where $W_2$ is the $2$-Wasserstein distance 
and the concentration scale $\eps$ is much smaller than the separation scale $\delta$, that is,
\[
\eps\ll \delta.
\]
Notice that the Wasserstein distance is well-defined because, by \eqref{C4}, $\bar\omega_i/a_i$ and $\delta_{\bar Y_i}$ are both probability measures. For a comprehensive  introduction into Wasserstein distances (and the theory of optimal transportation in general), we refer to Villani's monograph \cite{Villani03} and Chapter 7 therein.  Notice that if one of the marginals is an atomic measure as in \eqref{eq23}, the Wasserstein distance reduces to a simple second moment function
\begin{equation}\label{C1}
W_2\left(\frac{\bar{\omega}_i}{a_i},\,\delta_{\bar{Y}_i}\right) =\left(\frac1{a_i}  \int_{\Omega} |x-\bar Y_i|^2 \bar \omega_i\, dx\right)^{\frac12}.
\end{equation}
Considering this expression as a function of $\bar Y_i$, it is easily seen that the Wasserstein distance is minimized by locating $\bar Y_i$ at the center of vorticity. Indeed, setting
\begin{equation*}
\bar{X}_i=\frac{1}{a_i}\int_{\Omega} x\,\bar{\omega}_i(x)\,dx,
\end{equation*}
it holds that
\begin{equation}\label{C6}
W_2\left(\frac{\bar{\omega}_i}{a_i},\,\delta_{\bar{X}_i}\right)\le\inf\limits_{\bar Y_i}W_2\left(\frac{\bar{\omega}_i}{a_i},\,\delta_{\bar Y_i}\right).
\end{equation}

We suppose that the intensities $a_i$,  the circulations $\gamma_m$ and the separation scale $\delta$ are independent of the concentration scale $\eps$. Then \eqref{eq23} means that, up to rescaling with the scale independent constant $a_i$,  the vortex patch $\bar \omega_i$ approximates the Dirac measure $\delta_{\bar Y_i}$ if $\eps\ll1$. A prototype vorticity field   is thus a Dirac sequence or a constant vortex patch of the form $ \varepsilon^{-2}\chi_{B_\varepsilon(\bar Y_i)}$.  In this paper, we consider unbounded perturbations of such sequences in the sense that
$\bar \omega = \bar \omega^p + \bar \omega^{\infty}$ and 
\begin{equation}\label{eq6}
\|\bar{\omega}^{p}\|_{L^p}\lesssim\frac{1}{\varepsilon^{2(1-\frac{2}{p})}}, \quad \|\bar{\omega}^{\infty}\|_{L^\infty}\lesssim\frac{1}{\varepsilon^2}
\end{equation}
for some $p>2$. Here and in the following,  $A\lesssim B$ means that 
$A$ is bounded by $B$ up to a multiplicative constant independent of $\varepsilon$ and $t$. Notice that these perturbations are \emph{small} in the sense that $\|\varepsilon^{-2}\chi_{B_\varepsilon(\bar Y)}\|_{L^p}\sim \eps^{2(\frac1p-1)}\gg\eps^{2(\frac2p-1)}\gtrsim \|\bar{\omega}^{p}\|_{L^p}$. We also remark that, as a consequence of the DiPerna--Lions theory, if $\omega_i^p$, $\omega_i^\infty$ are defined as the solutions to the linear transport equation with velocity $u$ and initial datum  $\bar{\omega}_i^p$ and $\bar{\omega}_i^\infty$, respectively, so that $\omega_i=\omega_i^p+\omega_i^\infty$ by the uniqueness for the linear equation, we have thanks to the renormalization property that
\[
\|\omega_i^p(t)\|_{L^p}=\|\bar{\omega}_i^p\|_{L^p}\le\|\bar{\omega}^p\|_{L^p}
\]
and
\[
\|\omega_i^\infty(t)\|_{L^\infty}=\|\bar{\omega}_i^\infty\|_{L^\infty}\le\|\bar{\omega}^\infty\|_{L^\infty}.
\]
Thus the bound imposed in \eqref{eq6} on the initial datum carries over to the solution $\omega_i(t)$.


Our main goal in this paper is to establish a rigorous link between the Euler equations and the Kirchhoff--Routh point vortex system
\begin{equation}\label{eq18}
    \begin{cases}
    \frac{dY_i}{dt}(t)=\sum_{j\neq i}a_j K(Y_i(t)-Y_j(t))+\nabla^\perp\theta(t,Y_i(t)) \\
    \quad\quad\quad\quad\quad\quad+\sum_{m=1}^{M} \left(\sum_{j=1}^{N} a_j w_m(Y_j)+\gamma_m\right) \xi_m(Y_i), \\
    Y_i(0)=\bar{Y}_i,
    \end{cases}	
\end{equation}
where $\theta$ represents the interaction with the boundary and is defined as the solution to the Laplace equation
\begin{equation*}
	\begin{cases}
	-\Delta\theta(t,\cdot)=0 &\text{in }\Omega, \\
	\theta(t,\cdot)=\sum_{j}a_j G(\cdot-Y_j(t)) &\text{on }\partial\Omega.
	\end{cases}
\end{equation*}
Upon choosing $T$ smaller if necessary, we will furthermore assume that
\begin{equation}\label{C12}
\min\limits_{i\neq j}|Y_i(t)-Y_j(t)|\ge\frac{\delta}2, \quad  \min\limits_{i}\text{dist}\big(Y_i(t),\partial\Omega\big)\ge\frac{\delta}2 
\end{equation}
for all $t\in [0,T)$. Notice that for certain initial configurations, a collapse of vortices is possible. Choosing $T$ with the above condition, however, avoids this scenario. We refer to the nice survey paper \cite{Aref07} for a discussion, see also \cite{FlucherGustafsson97} for stability analyses.

Under   the scaling estimate \eqref{eq6}, we show that if the vorticity initially concentrates around the points  $\bar Y_1,\dots,\bar Y_N$  in the sense of \eqref{eq23}, then at any later time $t\in[0,T)$, the vorticity concentrates around the solution $Y_1(t),\dots ,Y_N(t)$ of the Kirchhoff--Routh system \eqref{eq18}.

\begin{theorem}\label{th2}
Let $T$ be given such that \eqref{C5} and \eqref{C12} hold. Let $\bar Y_1,\dots,\bar Y_N\in\Omega$ be such that \eqref{eq23} holds
for any $i\in\{1,\dots,N\}$ and suppose that $Y_1,\dots,Y_N$ solve the Kirchhoff--Routh system. Then $T$ is independent of $\eps$, i.e, $T\gtrsim 1$, and there exists a constant $C<\infty$ independently of $\eps$ such that, for any $i\in \{1,\dots,N\}$, 
\begin{equation}\label{C2}
W_2\left(\frac{\omega_i(t)}{a_i} ,\delta_{Y_i(t)}\right) \lesssim e^{Ct}\eps\quad \mbox{for all }t\in [0,T).
\end{equation}
\end{theorem}

The result in Theorem \ref{th2} translates into an estimate between the full vorticity field $\omega = \sum_i \omega_i$ and the empirical measure associated with the point vortex system. Indeed, interpreting the Kantorovich--Rubinstein distance $W_1$ as the dual norm $W^{-1,1}$, that is,
\[
W_1(f,g) = \sup\left\{\int_{\Omega} (f-g)\zeta\, dx: \|\grad \zeta\|_{L^{\infty}}\le 1\right\},
\]
cf.~\cite[Theorem 1.14]{Villani03}, the distance function can be readily extended to a distance between two not necessarily nonnegative functions of equal mean by setting
\[
W_1(f,g) = W_1((f-g)_+,(f-g)_-),
\]
where the subscript plus and minus signs indicate the positive and negative parts of a function. We then have the following estimate.

\begin{cor}\label{cor2}
Under the assumptions of Theorem \ref{th2}, it holds that
\[
W_1\left(\omega(t),\sum_{i=1}^N a_i \delta_{Y_i(t)}\right) \lesssim  e^{Ct}\eps\quad\mbox{for all }t\in[0,T).
\]
\end{cor}
In particular, because Kantorovich--Rubinstein distances metrize weak convergence, cf.~\cite[Theorem 7.12]{Villani03}, the result can be interpreted as an estimate on the order of weak convergence of the vorticity field $\omega(t)$ towards the empirical measure $\sum_i a_i \delta_{Y_i(t)}$: For any $t\in[0,T)$,  
\[
\omega(t) \longrightarrow \sum_{i=1}^N a_i \delta_{Y_i(t)}\quad \mbox{weakly with order at most }\eps.
\]
Here, weak convergence has to be understood in the sense of weak convergence of measures. We recall that our result holds true for \emph{any} solution to the Euler equation. Thus, in the event that it turns out that the two-dimensional Euler equations are not uniquely solvable, we regain uniqueness in the singular limit $\eps\to 0$.

Alternatively, we can express the estimate in terms of the centers of vorticity
\[
X_i(t) = \frac1{a_i}\int_{\Omega} x\omega_i(t)\, dx
\]
in the following way:

\begin{theorem}\label{th3}
Under the assumptions of Theorem \ref{th2}, for any $i\in \{1,\dots,N\}$, it holds that	 
	\begin{equation*}
		|X_i(t)-Y_i(t)|\lesssim e^{C\,t}\,\varepsilon \quad\mbox{and}\quad\left|\frac{d}{dt}X_i(t)-\frac{d}{dt}Y_i(t)\right|\lesssim e^{C\,t}\,\varepsilon
	\end{equation*}
for all $t\in [0,T)$.
\end{theorem}
Hence, both  position and  velocity of the centers of vorticity deviate from those of the point vortex system \eqref{eq18} only by a constant of order $\eps$ as $\eps\ll1$.

We add a comment on a possible generalization of the previous results.

\begin{remark}
Keeping track on how the constants in  our estimates depend on the separation scale $\delta$, we find that $T\gtrsim \delta^2$ and the constants $C$ in the exponential rates  in our main results grow as $1/\delta^2$. The results in Theorems \ref{th2} and \ref{th3} and Corollary \ref{cor2} can thus be generalized via iteration to the setting in which $\delta$ in \eqref{C5} and \eqref{C12} is arbitrarily small but finite and independent of the initial separation scale in \eqref{C9}. As a consequence, our result applies essentially up to the time at which vortex patches or the idealized point vortices collide.
\end{remark}

We also remark that the Kirchhoff--Routh system \eqref{eq18} features leapfrogging in  simple geometric situations, for instance,  if $\Omega$ is a ball. Indeed, if we place two vortices of equal sign (for simplicity) into this ball and both vortices are located sufficiently close to each other, the vortices start spinning around each other while traveling along the domain boundary. Conditions for leapfrogging in the case of the half-plane were already computed by Hicks \cite{Hicks22}. (In fact, Hicks studies the problem of four vortices in $\R^2$, whose location is  symmetric with respect to one axis. Therefore, Hicks computations also apply to the half-plane problem, if the two vortices outside the half-plane are treated as mirror vortices.) To the best of our knowledge, the present paper is the first to rigorously derive leapfrogging dynamics for the Euler equation. Yet, the available techniques seem not to be sophisticated enough in order to study the much more interesting leapfrogging problem for vortex rings, that we mentioned in the introduction.

We conclude this section with a comment on possible extensions to the vortex-wave system introduced in \cite{MarchioroPulvirenti91_2}.	In the vortex-wave system, point vortices coexist with a smoother background vorticity, and the corresponding velocity is generated by both components. Convergence from the  Euler vorticity equations to the vortex-wave system was proved in \cite{MarchioroPulvirenti93} using the same techniques that led to the convergence towards the standard point vortex model. We believe that our method would allow for an analogous \emph{qualitative} result in our setting. However, it seems unlikely that we obtain any reasonable quantitative results, since to prove convergence of the background vorticity more general stability estimates for the Euler equations would be needed. We recall that these are currently available only in the Yudovich class of bounded vorticities (cf.~\cite{Loeper06}) and lead also in other situations only to presumably suboptimal estimates, see, e.g., \cite{Chemin96,Seis20}.

The remainder of the article is devoted to the proofs.

\section{Proofs}\label{S3}

We first show how Theorem \ref{th2} implies Corollary \ref{cor2}.

\begin{proof}[Proof of Corollary \ref{cor2}] We apply metric properties of the Kantorovich--Rubinstein distance $W_1$ to the effect that
\begin{equation*}
    \begin{split}
    W_1\left(\omega(t),\sum_i a_i\delta_{Y_i(t)}\right) &\le \sum_i W_1(\omega_i (t) ,a_i \delta_{Y_i}) = \sum_i |a_i| W_1\left(\frac{\omega_i(t)}{a_i},\delta_{Y_i(t)}\right) \\
    & = \sum_i |a_i| \int_{\Omega} |x-Y_i(t)|\,\frac{\omega_i(t,x)}{a_i}\,dx.
    \end{split}
\end{equation*}
It remains to use Jensen's inequality to observe that $W_1\le W_2$ and the statement of the corollary follows from Theorem \ref{th2}.
\end{proof}


We will now turn to the proofs of Theorems \ref{th2} and \ref{th3} simultaneously. Our overall strategy is strongly inspired by that of Marchioro and Pulvirenti \cite{MarchioroPulvirenti83}. However, in order to be able to derive the statement for \emph{any} solution (recall that uniqueness of solutions is not known in the regularity setting under consideration), we will not follow the regularization procedure performed in the original paper, but we approach the problem in a more direct way. Yet, most of the individual lemmas that we derive in the following have their analogues in \cite{MarchioroPulvirenti83}.

Notice that in view of \eqref{C6}, we may always assume that \eqref{eq23} holds with $\bar Y_i$ replaced by $\bar X_i$. 
We will first establish concentration estimates around the centers of vorticity.
We start by introducing some notation.   

Using the vortex patch decomposition $\omega=\sum_{i=1}^N \omega_i$ in the Biot--Savart law \eqref{eq26}, the fluid velocity can be written as
\begin{equation*}
	u=\sum_{i=1}^N K*\omega_i + \nabla^\perp\eta  +\sum_{m=1}^M\left(\int_{\Omega} w_m\omega\,dz+\gamma_m\right)\xi_m =:\sum_{i=1}^N u_i +u^b,
\end{equation*}
where $u_i:=K*\omega_i$ is the velocity generated by the $i$th patch, and $u^b:=\nabla^\perp\eta+u^h$ the velocity generated by the interaction with the boundary, in which 
\begin{equation*}
	u^h:=\sum_{m}\left(\int_{\Omega} w_m \omega \,dz+\gamma_m\right)\xi_m
\end{equation*}
is the term due to the presence of the holes. We furthermore write $\Omega_i(t) := \spt\omega_i(t)$ and $\bar{\Omega}_i := \spt\bar{\omega}_i$.

 Our first concern is a control of the velocity field generated by the $j$th patch in $\Omega_i(t)$.

\begin{lemma}\label{lem1}
	Let $i\in\{1,\ldots,N\}$ be given. Then for any $j\neq i$ it holds that
	\begin{equation*}
		\|u_j(t)\|_{C^{0,1}(\Omega_i(t))}\lesssim 1 \quad\mbox{for any } t\in[0,T).
	\end{equation*}
\end{lemma}
\begin{proof}
 Let us prove first the $C^0$ bound. By the definition of $T $ in \eqref{C5}, we have that $|x-y|\ge\delta/2$ for any $x\in\Omega_i(t)$ and $y\in\Omega_j(t)$, and thus, by the definition of $u_j$ and $K$,
	\begin{equation*}
		|u_j(t,x)|\le\frac{1}{\pi\delta}\int_{\Omega}|\omega_j(t,y)|\,dy=\frac{|a_j|}{\pi\delta}\sim 1.
	\end{equation*}
	
Now to the Lipschitz bound. Since $K$ is  Lipschitz on $B_{\delta/2}(0)^c$ with Lipschitz constant of order $\delta^{-2}$, we have for any $x,\,z\in\Omega_i(t)$,
	\begin{equation*}
		|u_j(t,x)-u_j(t,z)|\le \int_{\Omega}\left|K(x-y)-K(z-y)\right| |\omega_j(t,y)|\,dy \lesssim \frac{|a_j|}{\delta^2}|x-z|\sim |x-z|.
	\end{equation*}
This concludes the proof.
\end{proof}

We now show that the velocity field induced by the boundary interaction is bounded uniformly in the support of the vortex patch.

\begin{lemma}\label{lem2}
	Let $i\in\{1,\ldots,N\}$ be given. Then
	\begin{equation*}
		\|u^b(t)\|_{C^{0,1}(\Omega_i(t))}\lesssim 1 \quad\mbox{for any } t\in[0,T).
	\end{equation*}
\end{lemma}
\begin{proof}
	Because the support of $\omega$ has a distance at least $\delta/2$ to the boundary, it is clear that $G\ast\omega$  
	is smooth on $\partial\Omega$. 
 Moreover for $x\in\partial\Omega$, it holds that
		\begin{equation*}
			|G*\omega(x)|\lesssim\max\left\{\left|G\left(\frac{\delta}{2}\right)\right|,|G(\text{diam}\,\Omega)|\right\}\lesssim 1.
		\end{equation*}
By the maximum principle for harmonic functions, we deduce that $\|\eta\|_{L^\infty(\Omega)}\lesssim 1$, and, by standard estimates for harmonic functions, this bound carries over to the gradient on any compact  subset of $\Omega$.  
	In particular, by the definition of $T$ in \eqref{C5}, 
	\begin{equation*}
		\|\nabla\eta\|_{L^\infty(\Omega_i(t))}\lesssim 1.
	\end{equation*}

Let us turn our attention to the velocity contribution  $u^h$ generated by the inner holes. We have
    \begin{equation*}
    	\sum_{m=1}^{M}\left| \int_{\Omega}w_m\omega\,dx \right|\le \|w_m\|_{L^\infty}\,\|\omega\|_{L^1}\lesssim 1,
    \end{equation*}
    because the harmonic measures $w_m$ are bounded independently of $\varepsilon$ by the maximum principle for harmonic functions. Since the same is also true for the vector fields $\xi_m$ and, by assumption, the circulations $\gamma_m$, there holds
    \begin{equation*}
    	\|u^h\|_{L^\infty(\Omega_i(t))}\lesssim 1.
    \end{equation*}
    
    Moreover, for $x,y\in\Omega_i(t)$ we have
    \begin{equation*}
    	|u^h(x)-u^h(y)|\le \sum_{m=1}^{M}\left| \int_{\Omega}w_m\omega\,dx +\gamma_m\right| |\xi_m(x)-\xi_m(y)|\lesssim |x-y|,
    \end{equation*}
    because of the smoothness of the $\xi_m$'s on the compact set $\bar{\Omega}_i(t)$. This completes the proof.

\end{proof}

These two first lemmas give us bounds on those velocity contributions that are   generated by far vortex patches and  the interaction with the boundary (the so called far-field), but we still know nothing about the velocity generated by the patch itself (the so called near-field). What we have for the moment, however, is already sufficient to prove concentration in terms of the $2$-Wasserstein distance around the centers of vorticity $X_i(t)$. 
Before stating the next lemma, we compute their velocity. Notice first that there is no self-induced motion,
\[
\int_{\Omega} u_i(t,x)\,\omega_i(t,x)\,dx = \iint_{\Omega\times\Omega} K(x-y)\omega_i(x)\omega_i(y)\, dxdy = 0,
\]
because the Biot--Savart kernel on $\R^2$ is odd,  that is, $K(z)=-K(-z)$. Hence, a direct computation reveals that
\begin{equation}\label{eq19}
	\frac{d}{dt}X_i(t)=\frac{1}{a_i}\int_{\Omega} u(t,x)\,\omega_i(t,x)\,dx  =\frac{1}{a_i}\int_{\Omega} F_i(t,x)\,\omega_i(t,x)\,dx  ,
\end{equation}
where
\begin{equation*}
	F_i(t,x):=\sum_{j\neq i}u_j(t,x)+u^b(t,x)=u(t,x)-u_i(t,x)
\end{equation*}
is the velocity far-field associated with the $i$th vorticity patch.

We now turn to the key concentration lemma, that was already found in \cite{MarchioroPulvirenti83}.

\begin{lemma}[\cite{MarchioroPulvirenti83}]\label{lem3}
	Let $i\in\{1,\ldots,N\}$  be given. There exists a constant $C<\infty$ dependent only on $\delta$ such that
	\begin{equation*}
    		W_2\left(\frac{\omega_i(t)}{a_i},\,\delta_{X_i(t)}\right)\le e^{C\,t}\,\varepsilon \quad\mbox{for any  }t\in[0,T).
	\end{equation*}
\end{lemma}
\begin{proof}
We recall from \eqref{C1} that the $2$-Wasserstein distance reduces to a simple second moments function if one of the marginals is an atomic measure, that is
	\begin{equation*}
		W_2(t): = W_2\left(\frac{\omega_i(t)}{a_i},\,\delta_{X_i(t)}\right)=\left(\frac{1}{a_i}\int_{\Omega}|x-X_i(t)|^2\,\omega_i(t,x)\,dx\right)^{1/2}.
	\end{equation*}
Computing its time derivative, we get (forgetting about the $t$'s),
	\begin{equation*}
		\begin{split}
		\frac{d}{dt}W_2^2&=\frac{2}{a_i}\int_{\Omega} (x-X_i)\cdot \left(u(x)-\frac{dX_i}{dt}\right)\,\omega_i(x)\,dx \\
		&=\frac{2}{a_i}\int_{\Omega} (x-X_i)\cdot u_i(x)\,\omega_i(x)\,dx+\frac{2}{a_i}\int_{\Omega} (x-X_i)\cdot \left(F_i(x)-\frac{dX_i}{dt}\right)\,\omega_i(x)\,dx.
		\end{split}
	\end{equation*}
	The first integral vanishes because $K$ is odd, while into the second one we plug the expression for the derivative of $X_i$ in \eqref{eq19},
	\begin{equation*}
		\frac{d}{dt}W_2^2=\frac{2}{a_i^2}\iint_{\Omega\times\Omega}(x-X_i)\cdot \big(F_i(x)-F_i(y)\big)\,\omega_i(x)\,\omega_i(y)\,dy\,dx.
	\end{equation*}
	Since we are considering times $t\le T$, from Lemmas \ref{lem1} and \ref{lem2} we know that $F_i$ is  Lipschitz on $\Omega_i(t)$. Therefore, 	 assuming that $\omega_i\ge0$ for notational simplicity,
	\begin{equation*}
		\begin{split}
		\left|\frac{d}{dt}W_2^2\right|&\lesssim  \iint_{\Omega\times\Omega}|x-X_i|\,|x-y|\,\omega_i(x)\omega_i(y)\,dy\,dx \\
		&\lesssim \int_{\Omega}|x-X_i|^2\,\omega_i(x)\,dx+ \iint_{\Omega\times\Omega} |x-X_i|\,|y-X_i| \omega_i(x)\omega_i(y)\,dy\,dx \\
		&\lesssim W_2^2,
		\end{split}
	\end{equation*}
	where we also used the triangle in the second and Jensen's inequality in the third estimate. Using a Gronwall argument and keeping in mind that by hypothesis \eqref{eq23} and \eqref{C6} the initial Wasserstein distance is bounded by $\varepsilon$, we obtain our thesis.
\end{proof}

Lemma \ref{lem3} gives us the concentration result in terms of $\varepsilon\ll 1$, if we can ensure that the time $T=T_\varepsilon$ stays bounded away from $0$, i.e., $T\gtrsim 1$,  as $\varepsilon$ becomes small. To do so, we need to bound also the near-field $u_i$, so that we can estimate the velocity with which the patch $\Omega_i$ moves. We begin with a result that allows us to estimate the near velocity field of a point on the boundary of the patch in terms of its distance from the center of vorticity. 

\begin{lemma}\label{prop1}
	Let $i\in\{1,\ldots,N\}$ and $x\in\partial\Omega_i(t)$ be given. There exists a constant $C<\infty$ dependent only on $\delta$ such that
	\begin{equation}\label{C8}
		|u_i(t,x)|\lesssim 1+\frac{e^{Ct}}{|x-X_i(t)|},
	\end{equation}
	for any $t\in [0,T)$.
\end{lemma}
\begin{proof}
	Since the time $t$ is fixed, in this proof we will frequently omit it. Setting
	\begin{equation*}
		R:=|x-X_i|,
	\end{equation*}
	we have, assuming again that $\omega_i$ is nonnegative for notational convenience, 
	\begin{equation*}
		\begin{split}
		|u_i(x)|&\le \frac{1}{2\pi}\int_{\Omega}\frac{1}{|x-y|} \omega_i(y)\,dy\\
		& \lesssim \int_{B_{R/2}(X_i)}\frac{1}{|x-y|}\omega_i(y)\,dy  +\int_{B_{R/2}(X_i)^c}\frac{1}{|x-y|}\omega_i(y)\,dy =:I_1+I_2.
		\end{split}
	\end{equation*}
	The term $I_1$ is easily bounded by $2\,|a_i|/R$, because on its domain of integration it holds that $|x-y|\ge R/2$. To bound the other  term, we make use of the interpolation-type estimate
	\begin{equation}\label{eq25}
		\int_A \frac{1}{|x-y|}\,\omega_i(y)\,dy\lesssim \ell^{1-\frac{2}{p}} \|\omega_i^p\|_{L^p(A)} + \ell \|\omega_i^\infty\|_{L^\infty(A)} + \frac{1}{\ell} \|\omega_i\|_{L^1(A)},
	\end{equation}
	that holds true for any measurable subset $A$ of $\Omega$ and any $\ell>0$. Notice that this is a refinement of the estimate
	\begin{equation}\label{C7}
	\int_A \frac1{|x-y|}\omega_i(y)\, dy \lesssim \|\omega_i\|_{L^p(A)}^{\frac{p}{2(p-1)}}\|\omega_i\|_{L^1(A)}^{\frac{p-2}{2(p-1)}},
	\end{equation}
	that holds for $p>2$ and that, to the best of our knowledge, has been first derived in \cite{Iftimie99}. Indeed, choosing $\omega_i^\infty=0$ in \eqref{eq25} and optimizing in $\ell$ yields \eqref{C7}. 
	We postpone the simple proof of \eqref{eq25} and proceed with the estimate of $I_2$. Keeping in mind Lemma \ref{lem3}, we have
	\[
	\|\omega_i\|_{L^1(B_{R/2}(X_i)^c)} \lesssim \frac{1}{R^2} W_2\left(\frac{\omega_i}{a_i},\delta_{X_i}\right)^2\lesssim \frac{\eps^2}{R^2}e^{Ct},
	\]
	and thus, by assumption \eqref{eq6}, 
	\begin{equation*}
	    I_2\lesssim  
	     \frac{\ell^{\left(1-\frac{2}{p}\right)}}{\eps^{2\left(1-\frac{2}{p}\right)}} + \frac{\ell}{\eps^2} + \frac{\eps^2}{\ell\,R^2} e^{Ct}.
	\end{equation*}
	Setting $\ell=\eps^2/R$, the statement in \eqref{C8} follows.
	
	It remains to provide the argument for \eqref{eq25}, in which we drop the index $i$ for convenience.   We  decompose
	\begin{align*} 
\int_A \frac{1}{|x-y|}\,\omega(y)\,dy
		&=\int_{A\cap B_\ell(x)} \frac{1}{|x-y|}\,\omega^p(y)\,dy \\
		&\qquad+ \int_{A\cap B_\ell(x)} \frac{1}{|x-y|}\,\omega^\infty(y)\,dy + \int_{A\setminus B_\ell(x)} \frac{1}{|x-y|}\,\omega(y)\,dy.
	\end{align*}
	The first two integrals can be estimated with the help of the H\"older inequality ($p'$ being the conjugate exponent to $p$) and a change of variables,
	\[
	\int_{A\cap B_\ell(x)} \frac{1}{|x-y|}\,\omega^p(y)\,dy \le \left(\int_{B_\ell(0)}\frac{1}{|y|^{p'}}\,dy\right)^\frac{1}{p'} \|\omega^p\|_{L^p(A)}\lesssim \ell^{1-\frac{2}{p}} \|\omega^p\|_{L^p(A)}
	\]
	because $p>2$, and similarly for the second integral. For the third integral, we simply observe that
	\[
	\int_{A\setminus B_\ell(x)} \frac{1}{|x-y|}\,\omega(y)\,dy \le \frac{1}{\ell}\|\omega\|_{L^1(A)}.
	\]
	This concludes the proof.
\end{proof}

\begin{remark}
	We remark here that a small modification of the proof would yield the same control of the velocity field  if $\omega_i$ belonged to   $L^{\infty}((0,T);\text{BMO})$ with
	\begin{equation}\label{eq24}
		\|\omega_i(t)\|_{BMO}\lesssim \varepsilon^{-2}\quad\mbox{for all }t\in[0,T).
	\end{equation}
We omit this case here, because
it is in general not known if the $BMO$-norm is preserved in time, and therefore assumption \eqref{eq24} may not be attainable under general assumptions on the initial datum. We refer to \cite{BernicotKeraani14} for (optimal) estimates on the $\text{BMO}$ norm for the two-dimensional Euler equations.
 \end{remark}

We still need one more ingredient in order to bound $T$ from below uniformly in $\varepsilon$. With Lemma \ref{prop1} we have estimated the near velocity field on boundary points in terms of their distance from the center of vorticity. Should this distance vanish in $\varepsilon$, then this result would be useless. We need therefore to estimate also the velocity of the center of vorticity. In this way, we will be able to bound the velocity of all points on the boundary of the patch (and therefore of the patch itself): If the points stay away from the center, then we use Lemma \ref{prop1}, otherwise we use the bound on the center of vorticity.

\begin{lemma}\label{lem5}
	Let $i\in\{1,\ldots,N\}$ be given. Then
	\begin{equation*}
		|X_i(t)-\bar{X}_i|\lesssim t \quad\mbox{for any } t\in[0,T).
	\end{equation*}
\end{lemma}
\begin{proof}
	Keeping in mind the expression for the velocity of the vorticity centers \eqref{eq19}, we simply compute
	\begin{equation*}
		|X_i(t)-\bar{X}_i|=\left|\int_{0}^{t}\frac{d}{ds}X_i(s)\,ds\right|=\frac{1}{|a_i|}\left|\int_{0}^{t}\int_{\Omega} F_i(s,x)\,\omega_i(s,x)\,dx\,ds\right| 
		\lesssim  t,
	\end{equation*}
	because we know by Lemmas \ref{lem1} and \ref{lem2} that $F_i$ is bounded  on $\Omega_i(t)$ uniformly in $\eps$ for times $t\in [0,T)$ and $\|\omega_i\|_{L^1(\Omega)} = |a_i| $.
\end{proof}

We are now ready to bound $T$ from below uniformly in $\eps$.
\begin{lemma}\label{lem6}
	It holds that $		T\gtrsim 1$.
	\end{lemma}

The proof of Lemma \ref{lem6} is quite elementary.

\begin{proof} We may without loss of generality suppose that $T\lesssim 1$, because otherwise there is nothing to prove. Moreover, we may assume that $T$ is solely defined through \eqref{C5}, because any time implicitly determined by \eqref{C12} is independent of $\eps$ by definition.
In particular, one of the inequalities in \eqref{C5} has to be an equality, and thus, in view of the assumption on the initial data in \eqref{C9}, there are fluid particles carrying nonzero vorticity that are transported over a distance at least $\delta/4$ in the time interval $[0,T]$. We denote the minimal time in which a fluid particle moves over that distance by $T'$. More precisely, for every initial vortex patch $\bar{\Omega}_i$, we consider the $\delta/4$-neighborhood 
\begin{equation*}
\bar \Omega^{\delta}_i:=\left\{y\in \Omega:\: \dist(y,\bar{\Omega}_i)\le\delta/4\right\},
\end{equation*}
and define $T'=T'_\varepsilon$ as the first instant, when some patch $\Omega_i(t)$ touches $\partial \bar \Omega_i^{\delta}$. Hence, 
\begin{equation*}
T':=\min\left\{t\ge 0:\: {\phi_t(\bar \Omega_i )} \cap  \partial \bar \Omega_i^{\delta} \not=\emptyset \mbox{ for some } i\in\{1,\ldots,N\} \right\},
\end{equation*}
where $\phi$ is the flow associated with   $u$, cf.~\eqref{C13}.  Observe that $T'\le T$, and thus $T'\lesssim 1$ by assumption. We will  prove the lower bound $T'\gtrsim 1$, which is stronger than the statement of the lemma.

	Let $\varepsilon$ be arbitrarily fixed, and let $\Omega_i$ be the first patch that touches $\partial\bar\Omega_i^{\delta}$. Hence, since the flow is (Lipschitz) continuous in the time variable thanks to the boundedness of the velocity field, there exists an $x\in\partial\bar\Omega_i$ such that 
		\begin{equation}\label{eq11}
	|x-\phi_{T'}(x)|  \ge \frac{\delta}4.
		\end{equation}

	Let us divide the proof into two cases. 
	
	\emph{Case 1}. There exists $t\in[0,T']$ such that $|X_i(t)-\bar{X}_i|> \frac{\delta}{16}$. \\
	This means that $X_i(t)$ has covered a distance of at least $\delta/16$. In this case we use Lemma \ref{lem5}, obtaining
	\begin{equation*}
		\frac{\delta}{16}\lesssim t \le T'
	\end{equation*}
uniformly $\varepsilon$, which is what we have to prove.

	\emph{Case 2}. For every $t\in [0,T']$ it holds that $|X_i(t)-\bar{X}_i|\le\frac{\delta}{16}$. \\
	In this case, $X_i(t)$ has covered a smaller distance than  $\phi_{t}(x)$. Notice that because $\frac{\delta}4\le \dist(\bar \Omega_i,\partial\bar \Omega^{\delta}_i) \le |\phi_{T'}(x) - \bar X_i|$, we then have that
		\begin{equation}\label{eq15}
		|\phi_{T'}(x)-X_i(T')|\ge |\phi_{T'}(x)-\bar{X}_i|-|X_i(T')-\bar{X}_i|\ge \frac{\delta}{4}-\frac{\delta}{16}=\frac{3}{16}\,\delta.
	\end{equation}
We split our argument into two further subcases.
	
	\emph{Case 2.1}. For every $t\in [0,T']$ it holds that  $|\phi_t(x)-X_i(T')|>\frac{5 }{32}\delta$. \\
In this case we have
	\begin{equation}\label{eq12}
		|\phi_t(x)-X_i(t)|\ge |\phi_t(x)-X_i(T')|-|X_i(t)-X_i(T')|\ge \frac{5}{32}\,\delta-\frac{\delta}{8}=\frac{1}{32}\,\delta
	\end{equation}
	for every $t\in[0,T']$, because $|X_i(t)-X_i(T')|\le \delta/8$ from the hypothesis of case 2. Inequalities \eqref{eq11} and \eqref{eq12} yield the bound on $T'$. Indeed, thanks to Lemma \ref{prop1} and \eqref{eq12}, for every $t\in[0,T']$
	\begin{equation*}
		|u_i(t,\phi_t(x))|\lesssim 1+\frac{e^{C\,T'}}{|\phi_t(x)-X_i(t)|} \lesssim 1
	\end{equation*}
	(where we also used that $T'\lesssim 1$) and this, coupled with Lemmas \ref{lem1} and \ref{lem2}, gives
	\begin{equation*}
		|u(t,\phi_t(x))|\lesssim 1 \quad\mbox{for all } t\in[0,T'].
	\end{equation*}
	But then, using also \eqref{eq11} yields
	\begin{equation*}
		\frac{\delta}{4}\le|\phi_{T'}(x)-x|=\left|\int_{0}^{T'}u(t,\phi_t(x))\,dt\right|\lesssim T',
	\end{equation*}
	uniformly in  $\varepsilon$, as desired.
	
	\emph{Case 2.2}. There exists $t\in [0,T')$ such that		$|\phi_t(x)-X_i(T')|\le \frac{5}{32}\,\delta$.\\
	Consider the maximal time for which this happens, i.e.,
	\begin{equation*}
		T_1:=\max\left\{t\in[0,T'):\,|\phi_t(x)-X_i(T')|\le\frac{5}{32}\delta\right\}.
	\end{equation*}
From \eqref{eq15} it follows that $T_1<T'$ and 
	\begin{equation*}
		|\phi_{t}(x)-X_i(T')|\ge \frac{5}{32}\,\delta \quad\mbox{for all }  t\in[T_1,T'].
	\end{equation*}
	From this we infer that
	\begin{equation}\label{eq14}
		|\phi_t(x)-X_i(t)|\ge |\phi_t(x)-X_i(T')|-|X_i(t)-X_i(T')|\ge \frac{5}{32}\,\delta-\frac{\delta}{8}=\frac{\delta}{32} 
	\end{equation}
	for any $t\in[T_1,T']$, where we have used $|X_i(t)-X_i(T')|\le\delta/8$ by the general hypothesis of Case 2.
	Moreover,  using \eqref{eq15} again,
	\begin{equation}\label{eq16}
		|\phi_{T'}(x)-\phi_{T_1}(x)|\ge|\phi_{T'}(x)-X_i(T')|-|\phi_{T_1}(x)-X_i(T')|\ge \frac{3}{16}\,\delta-\frac{5}{32}\,\delta=\frac{\delta}{32}.
	\end{equation}
	From now on, we can conclude exactly like in Case 2.1 using inequalities \eqref{eq14} and \eqref{eq16} instead  of \eqref{eq12} and \eqref{eq11}, respectively, and working on the interval $[T_1,T']$ instead of $[0,T']$, to obtain
	\begin{equation*}
		\frac{\delta}{32}\lesssim(T'-T_1)\le T'
	\end{equation*}
	for this case.
	\end{proof}

Until now, we have established that the vorticity field remains concentrated around the center of vorticity during the evolution for  time intervals independent of $\eps$. In order to prove Theorems \ref{th2} and \ref{th3}, it remains to show that the centers of vorticity are $\eps$-close to the point vortices. 
For this observation, we have to establish a bound on the difference of the boundary contributions. 
In a first step, we turn to the velocities induced by the outer boundary.

\begin{lemma}\label{lem7}
	For any $t\in [0,T)$ and any $i\in\{1,\dots,N\}$
it holds that
	\begin{equation*}
		|\nabla\eta(t,Y_i(t))-\nabla\theta(t,Y_i(t))|\lesssim e^{C\,t}\,\varepsilon+\sum_{j}|X_j(t)-Y_j(t)| .
\end{equation*}
\end{lemma}
\begin{proof} 
	Let us drop the $t$'s in this proof. Consider $x\in\partial\Omega$, then
	\begin{equation*}
		\begin{split}
		|\eta(x)-\theta(x)|&=\left|\int_{\Omega} G(x-y)\,\omega(y)\,dy-\sum_{j}a_j G(x-Y_j)\right| \\
		&\le\sum_{j}\int_{\Omega}|G(x-y)-G(x-Y_j)|\,|\omega_j(y)|\,dy.
		\end{split}
	\end{equation*}
	Since we are considering times smaller than $T$, we know that  $|x-Y_i|\ge\delta/2$  and $|x-y|\ge \delta/2$ for any $y\in \spt\omega$. In particular, it holds that
	\begin{equation*}
		|G(x-y)-G(x-Y_j)|\le\|\nabla G\|_{L^\infty(B_{\delta/2}(0)^c)}\,|y-Y_j|\lesssim |y-Y_j|.
	\end{equation*}
	Using the triangular inequality we then have 
	\begin{equation*}
		|\eta(x)-\theta(x)|\lesssim \sum_{j}\int_{\Omega}|y-X_j|\,|\omega_j(y)|\,dy+ \sum_{j}|a_j|\,|X_j-Y_j|.
	\end{equation*}
	Bounding the integral in the first term of the right-hand side by the Jensen inequality and Lemma \ref{lem3}, we see that
	\begin{equation}\label{eq20}
		\int_{\Omega}|y-X_j|\,|\omega_j(y)|\,dy\le|a_j|\,W_2\left(\frac{\omega_j}{a_j},\,\delta_{X_j}\right)\le|a_j|\,e^{C\,t}\,\varepsilon.
	\end{equation}
	Hence
	\begin{equation*}
		\|\eta(t)-\theta(t)\|_{L^\infty(\partial\Omega)}\lesssim e^{C\,t}\,\varepsilon+\sum_{j}|X_j(t)-Y_j(t)| 
			\end{equation*}
	for any $t\in[0,T)$. On the other hand, because $\eta-\theta$ is harmonic, this bound carries over to all of $\Omega$ by the maximum principle. Standard gradient estimates for harmonic functions in the interior of $\Omega$ then yield the desired estimate.
	\end{proof}

It remains to treat the velocity contributions that are due to the inner boundary components.

\begin{lemma}\label{lem8}
	For any $t\in[0,T)$ and any $i\in\{1,\ldots N\}$ it holds that
	\begin{equation*}
	    \begin{split}
	    H&:=\left|\frac{1}{a_i}\int_{\Omega} u^h(x)\,\omega_i(x)\,dx-\sum_{m=1}^{M}\left(\sum_{j=1}^{N}a_j w_m(Y_j)+\gamma_m\right)\xi_m(Y_i) \right| \\
	    &\qquad\qquad\lesssim e^{Ct}\varepsilon+\sum_{j}|X_j(t)-Y_j(t)|.
	    \end{split}
	\end{equation*}
\end{lemma}
\begin{proof}
	Again, in this proof we forget about the dependence on time. For notational convenience, we write
	\begin{equation*}
		u^h=\sum_{m=1}^{M} \left(\int_{\Omega} w_m\,\omega\,dz+\gamma_m\right) \xi_m = \sum_{m=1}^{M} u_m^h.
	\end{equation*}	
	
	Let us first consider the difference
	\begin{align*}
		H_m&:=\left|\int_{\Omega}w_m\omega\,dz\,\xi_m(x)-\sum_{j=1}^{N}a_j w_m(Y_j)\,\xi_m(Y_i)\right| \\
		&\le\sum_{j=1}^{N}\left|\int_{\Omega}w_m\omega_j\,dz\,\xi_m(x)-a_j w_m(Y_j)\,\xi_m(Y_i) \right| \\
		&\le\sum_{j=1}^{N} \left|\int_{\Omega}w_m\omega_j\,dz\right| |\xi_m(x)-\xi_m(Y_i)|  \\
		&\qquad +\sum_{j=1}^{N} \left| \int_{\Omega}w_m\omega_j\,dz-a_j w_m(Y_j)\right||\xi_m(Y_i)|.
	\end{align*}
	We recall from the proof of Lemma \ref{lem2} that $|\int_{\Omega} w_m \omega_j\,dz|\lesssim 1$. In addition, we observe  that $\nabla w_m$ is bounded on $\Omega_j$ because it is far from the boundary of $\Omega$, and therefore $w_m$ is Lipschitz here. Hence
	\begin{equation*}
		\begin{split}
		\left|\int_{\Omega}w_m\omega_j\,dz-a_j w_m(Y_j)\right|&\le\int_{\Omega_j} |w_m(z)-w_m(Y_j)|\,|\omega_j(z)|\,dz \\
		&\lesssim \int_{\Omega} |z-Y_j|\,|\omega_j(z)|\,dz \\
		&\lesssim W_2\left(\frac{\omega_j}{a_j},\delta_{X_j}\right)+|X_j-Y_j| \\
		&\lesssim e^{Ct}\eps+|X_j-Y_j|,
		\end{split}
	\end{equation*}
	where we also used Jensen's inequality and the control on the Wasserstein distance in Lemma \ref{lem3}. Therefore, since $|\xi_m(Y_i)|\lesssim 1$, we have that
	\begin{equation*}
	    H_m\lesssim |\xi_m(x)-\xi_m(Y_i)|+e^{Ct}\varepsilon+\sum_{j=1}^{N}|X_j-Y_j|.
	\end{equation*}

We use this estimate in order to bound the velocity induced by the $m$th hole,
	\begin{align*}
	    \left| u_m^h(x)-\left(\sum_{j=1}^N a_j w_m(Y_j)+\gamma_m\right)\xi_m(Y_i) \right|&\le H_m+|\gamma_m|\,|\xi_m(x)-\xi_m(Y_i)| \\
	    & \lesssim|\xi_m(x)-\xi_m(Y_i)|+e^{Ct}\varepsilon+\sum_{j=1}^{N}|X_j-Y_j|.
	\end{align*}
    Finally, observing that $|\xi_m(x)-\xi_m(Y_i)|\lesssim |x-Y_i|$ for $x\in\Omega_i$, because $\nabla\xi_m$ is bounded   away from the boundaries, we have
    \begin{equation*}
    	\begin{split}
    	H&\lesssim \sum_{m=1}^{M}\int_{\Omega}\left| u_m^h(x)-\left(\sum_{j=1}^N a_j w_m(Y_j)+\gamma_m\right)\xi_m(Y_i) \right| |\omega_i(x)|\,dx \\
    	&\lesssim \sum_{m=1}^{M} \int_{\Omega}|x-Y_i|\,|\omega_i(x)|\,dx+e^{Ct}\varepsilon+\sum_{j=1}^{N}|X_j-Y_j| \\
    	&\lesssim e^{Ct}\varepsilon+\sum_{j=1}^{N}|X_j-Y_j|,
    	\end{split}
    \end{equation*}
  where we have used,  as before, Jensen's inequality and Lemma \ref{lem3}.
\end{proof}

We are now in the position to prove Theorems \ref{th2} and  \ref{th3}.

\begin{proof}[Proof of Theorems \ref{th2} and  \ref{th3}]

We start with an estimate on the rate of change of the distance of $X_i(t)$ and $Y_i(t)$. Using the velocity formula of the vorticity centers \eqref{eq19} and the definition of the point vortex system \eqref{eq18}, we find that
\begin{align*}
\MoveEqLeft  \left|\frac{d}{dt}X_i(t) - \frac{d}{dt}Y_i(t)\right|\\
& \le \sum_{j\not=i}\left| \frac1{a_i} \iint_{\Omega\times\Omega} K(x-y)\omega_i(x)\omega_j(y)\, dxdy - a_j K(Y_i(t) - Y_j(t))\right|\\
&\qquad + \left|\frac1{a_i} \int_{\Omega} \grad\eta(x)\omega_i(x)\, dx - \grad\theta(Y_i(t))\right| + H \\
& \le \frac1{|a_i|} \sum_{j\not=i} \iint_{\Omega\times\Omega} |K(x-y) - K(Y_i(t) - Y_j(t))| |\omega_i(x)||\omega_j(y)|\, dxdy\\
&\qquad + \frac1{|a_i|} \int_{\Omega} |\grad\eta(x)  - \grad\eta(Y_i(t))||\omega_i(x)|\, dx+ |\grad\eta(Y_i(t)) - \grad\theta(Y_i(t))|  + H ,
\end{align*}
where we have dropped most of the $t$'s for notational convenience.  Using the Lip\-schitz property of the Biot--Savart kernel away from the origin, the Lipschitz estimate of Lemma \ref{lem2} on $\eta$, the requirements \eqref{C5} and \eqref{C12} on $T$ and Lemmata \ref{lem7} and \ref{lem8}, we find that
\begin{align*}
\left|\frac{d}{dt}X_i(t) - \frac{d}{dt}Y_i(t)\right| 
& \lesssim \sum_{j=1}^N\int_{\Omega} |x - Y_j(t)| |\omega_j(x)|\, dx + e^{Ct}\eps + \sum_{j} |X_j(t) -Y_j(t)|.
\end{align*}
Using  the triangle and Jensen's inequalities and the concentration estimate from Lemma \ref{lem3}, we finally deduce
\[
\left|\frac{d}{dt}X_i(t) - \frac{d}{dt}Y_i(t)\right| \lesssim   e^{Ct}\eps + \sum_{j} |X_j(t) -Y_j(t)|.
\]
Since
\[
\frac{d}{dt}|X_i(t) - Y_i(t)|\le \left|\frac{d}{dt}X_i(t) - \frac{d}{dt}Y_i(t)\right| ,
\]
summing over $i$ and using a Gronwall argument yields 
\[
\sum_i |X_i(t) - Y_i(t)| \lesssim   e^{Ct}\eps + e^{Ct} \sum_i |\bar{X}_i - \bar{Y}_i| ,
\]
where the value of $C$ might have changed.
Notice that
\[
|\bar{X}_i - \bar{Y}_i| = \frac{1}{a_i}\int |\bar{X}_i - \bar{Y}_i|\,\bar{\omega}_i(x)\,dx,
\]
then by the triangle and Jensen's inequalities, \eqref{eq23} and \eqref{C6} it follows that $|\bar{X}_i - \bar{Y}_i|\lesssim  \eps$.
A combination of the previous bounds yields the full statement of Theorem~\ref{th3}. 

To derive Theorem \ref{th2}, we have to combine Theorem \ref{th3} and Lemma \ref{lem3}.
\end{proof}

\section*{Acknowledgement} 
The authors thank the anonymous referee for several suggestions that helped to  improve  the manuscript substantially. They are grateful to Helena Nussenzveig Lopes for pointing out references. The second author, moreover, acknowledges inspiring discussions with Bob Jerrard on the topic of this paper.  This work is funded by the Deutsche Forschungsgemeinschaft (DFG, German Research Foundation) under Germany's Excellence Strategy EXC 2044 --390685587, Mathematics M\"unster: Dynamics--Geometry--Structure.

\bibliography{euler}
\bibliographystyle{acm}

\end{document}